\documentclass[journal]{IEEEtran}
\usepackage{amsmath}
\usepackage{amsthm}
\usepackage{amsfonts}
\usepackage{amssymb}
\usepackage{gensymb}
\usepackage{graphicx}
\usepackage{epstopdf}
\usepackage{cite}
\usepackage{xr}
\usepackage{hyperref}
\usepackage{url}
\usepackage{breakurl}

\newtheorem{theorem}{Proposition}

\begin{document}
\title{A concise, approximate representation of a collection of loads described by polytopes}
\author{Suhail Barot~\IEEEmembership{Student Member, IEEE} and Josh A. Taylor~\IEEEmembership{Member, IEEE}
\thanks{This work was supported by the Natural Sciences and Engineering Research Council of Canada. S. Barot and J. A. Taylor are with the Department of Electrical and Computer Engineering, University of Toronto, Toronto, Ontario, M5S 3G4, Canada ({\tt\small suhail.barot@mail.utoronto.ca, josh.taylor@utoronto.ca}).}
}
\maketitle

\begin{abstract}
Aggregations of flexible loads can provide several power system services through demand response programs, for example load shifting and curtailment. The capabilities of demand response should therefore be represented in system operators' planning and operational routines. However, incorporating models of every load in an aggregation into these routines could compromise their tractability by adding exorbitant numbers of new variables and constraints.

In this paper, we propose a novel approximation for concisely representing the capabilities of a heterogeneous aggregation of flexible loads. We assume that each load is mathematically described by a convex polytope, i.e., a set of linear constraints, a class which includes deferrable loads, thermostatically controlled loads, and generic energy storage. The set-wise sum of the loads is the Minkowski sum, which is in general computationally intractable. Our representation is an outer approximation of the Minkowski sum. The new approximation is easily computable and only uses one variable per time period corresponding to the aggregation's net power usage. Theoretical and numerical results indicate that the approximation is accurate for broad classes of loads.
\end{abstract}


\section{Introduction}
Demand response (DR), the coordinated control of collections of flexible loads, can render great benefits to power systems and is recognized as an essential new source of flexibility for renewable integration~\cite{fercDRAM}. DR activities are now widely engaged in by third party companies companies, utilities, and system operators; recently the PJM system operator reported that its DR program saved over \$$650$ million during a single week in August 2013~\cite{nytimes_drsave}. Comprehensive surveys on DR are provided by~\cite{callaway11EL,dietrich2011side,siano2014demand}. In this paper, we refer to the entity controlling a collection of loads as a load aggregator.

System operators must integrate DR into their operational routines to fully leverage its capabilities. For example, multiperiod optimal power flow or unit commitment can be used to perform load-shifting using DR alongside energy storage~\cite{kirschen2009quant,khodaei2011SCUC}. This is challenging because the loads in DR programs are often small, diverse, and numerous; a typical aggregation may contain upwards of $10^6$ loads. Exactly representing the loads of multiple DR aggregations within multiperiod optimal power flow could add millions of new variables and constraints, making it computationally intractable~\cite{Oren_RE_deferrableloads,hao_thermostatic,nayyar_taylor}. Moreover, the individual load models may be known to the load aggregator but not the system operator.

To overcome these difficulties, load aggregators need concise models of their loads' aggregate characteristics, thus enabling them to share their capabilities with the system operator without describing every load individually. System operators can then straightforwardly incorporate such a model into tasks like multiperiod optimal power flow or unit commitment as they would a conventional resource like grid-scale storage. Because the model is concise, i.e., consisting of a small number of variables and constraints, it does not increase the difficulty of the system operator's tasks. We further discuss the role of DR and concise modeling within multiperiod optimal power flow in Section~\ref{sec:role}.

In this paper, we develop a concise, approximate representation for aggregations of loads modeled by convex polytopes, i.e., sets of linear constraints. Since we only deal with convex polytopes, we will henceforth omit the term `convex' and simply write `polytope'. The set-wise sum of two sets is called the Minkowski sum, and is computationally intractable even for polytopes. As observed in~\cite{hao_thermostatic,alizadeh_loadmodels}, the flexibility of an aggregation of polytopic loads is captured by the Minkowski sum, which we define in Section~\ref{sec: Msum}. Approximate Minkowski sums are an active research area, but most work focuses on the calculation of two and three-dimensional sums of highly complex polytopes as in~\cite{Agarwal_2DMinkowski} and~\cite{Varadhan_3DMinkowski}. In Section~\ref{sec:gen_polytopes}, we develop a novel outer approximation of the Minkowski sum, which is easily computable in polynomial-time. Our method is generally applicable regardless of dimension, and also results in a polytope in $\mathbb{R}^D$, where $D$ is the number of time periods. This makes it easy to incorporate into optimization routines for power system operations without sacrificing tractability.

A number of existing papers describe techniques for concisely modeling large collections of loads, which we now summarize. Work on this topic has been on-going since the 1980s beginning with~\cite{chong_loadmodel_1985} and more recently in~\cite{Callaway_wind,amgarcia_prob_draggr,fathy_robust}, which model the probability distribution of temperatures in spaces controlled by thermostatic loads using a partial differential equation. Thermostatic loads are a particular focus area within DR work as they represent almost 20\% of load in industrialized countries such as the U.S.~\cite{eere_bldgs}. In~\cite{perfumo_kofman_ward}, the authors model the control of a collection of thermostatic loads using a second-order LTI system and design a controller to achieve desired power outputs and then return the aggregate system to steady-state.

Our work is closely related to several recent papers that approximate a collection of loads as generalized energy storage. In~\cite{nayyar_taylor}, charging electric vehicles are modeled as deferrable loads, and analytical generalized storage expressions for their aggregate capabilities are obtained. In~\cite{alizadeh_loadmodels}, many types of loads are clustered and aggregated using generalized battery models whose parameters are found by summing over the loads in a cluster. Load aggregations are approximated as time-varying thermal batteries in~\cite{mathieu2013energy,mathieu2015arbitraging} and as generalized batteries in~\cite{hao_thermostatic}; the latter derives inner and outer generalized battery models to represent a collection of thermostatic loads. The storage models obtained in these papers consist of linear constraints, similar to the polytope-based framework employed in this paper.

The rest of the paper is organized as follows. In Section~\ref{sec:load_models}, we give some general background on polytopes, the role of flexible loads aggregations in power system operations, and survey several common load types and their standard representations as polytopes. In Section~\ref{mainresult}, we present our approximation. Finally, we show that the approximation is exact for certain special classes of loads and present numerical results demonstrating the accuracy of the approach for more general load classes in Section~\ref{sec:examples}.

\section{Background}

\subsection{Notation}\label{sec:notation}

A polytope is a set in $\mathbb{R}^D$ whose boundary is composed of flat surfaces called facets~\cite{ziegler1995lectures}. These facets are derived from hyperplanes and are sets in $\mathbb{R}^{D - 1}$. We denote polytopes using bold script with subscripts for differentiation between them, e.g., $\mathbf{P}_1, \mathbf{P}_2, \dots, \mathbf{P}_k$. We restrict our attention to polytopes that are closed and bounded, i.e., compact.

The points within a polytope can be represented as convex combinations of the extreme points of the polytope~\cite{boyd_convex}. We denote points (or vectors) using lowercase italicized letters, e.g., $x, y$. The set of vertices of such polytopes then form a minimal unique (up to ordering) representation for a polytope. Such a representation is referred to as the V-representation of a polytope. Sets of vertices are denoted using uppercase letters with a bar, e.g., $\bar{X}, \bar{Y}$.

An alternate representation for a polytope is as the intersection of a collection of half-spaces (referred to as the H-representation of a polytope). In the H-representation, each half-space generates a facet of the polytope and is represented as a linear inequality, e.g., $ a^T x \leq b$. A minimal H-representation contains only inequalities corresponding to facets of the polytope with non-zero area, and is unique up to ordering and scaling. The H-representation is generally preferred to the V-representation for DR because it is the form of almost all load models.

We use uppercase letters to represent matrices and subscripts to indicate that a set or matrix is associated with a particular polytope. We may write the H-representation of a polytope in matrix form as  $ A_{1} x \leq b_{1} $, and denote it by the matrix-vector pair $(A_{1}, b_{1})$. The polytope can also be written explicitly as $\mathbf{P}_1 = \{x \, | \,  A_{1} x \leq b_{1} \}$ We use the term \textit{A-matrix} to refer to the matrix $A_{1}$ of a polytope in H-representation.

{\bf Example 1:} Consider a triangle in $\mathbb{R}^2$. In, V-representation, we may denote it by its set of vertices as $\bar{X}_{1} = \{(0,0), (1,0), (0,1)\} $. In H-representation, we may denote it by the matrix-vector pair $(A_{1}, b_{1})$, where
\[ A_{1} = \begin{bmatrix}
-1 & 0 \\
0 & -1 \\
1 & 1\\
\end{bmatrix}, 
b_{1} = \begin{bmatrix}
0 \\
0 \\
1
\end{bmatrix}. \]
Note how in this case, the vertices of the polytope are generated by solving the equalities associated with each inequality. In general, the vertices of a polytope will be generated from the solution of equalities associated with adjacent facets.

V-Representations and H-representations of a polytope can be derived from each other. Conversion from the H-representation to the V-representation is known as vertex enumeration; the reverse problem is known as facet enumeration. Unfortunately both of the above problems are, in general, NP-hard~\cite{boros_nphard}. For polytopes that are bounded, the complexity of vertex and facet enumeration remains open~\cite{khachiyan_generating_hard}. No tractable solutions to these problems are currently known.

Additionally, while the above refers to minimal V-representations and H-representations, both may contain redundant information. In the V-representation, this implies the inclusion of points lying inside the polytope. In the H-representation, this implies the inclusion of non-binding inequalities (i.e. inequalities that do not generate a facet of the polytope as their associated hyperplanes either lie outside the polytope or are tangent to it at a single point). Testing a component of either representation for redundancy can be done with linear programming~\cite{Fukuda_polyhedral_computation_faq}.

\subsection{Role within power system operations}\label{sec:role}
Large aggregations of flexible loads are valuable resources for power system operators and hence should be represented in power system dispatch routines. Multi-period optimal power flow is a standard approach to dispatching power systems with dynamic constraints such as ramping and storage capacity limits~\cite{taylor2015COPS}. Since DR also has dynamic constraints such as keeping the temperature of a building within a fixed range (a thermostatic load) and arrival and departure times (an electric vehicle), DR should also be represented within multi-period optimal power flow.

A simple instance of multi-period optimal power flow is given by
\begin{displaymath}
\begin{array}{rl}\displaystyle
\min_{p,\theta} & F(p)\\
\textrm{s.t.}
&\displaystyle \underline{p}_{i}(t)\leq p_{i}(t)\leq \overline{p}_{i}(t),\\
&\displaystyle p_{i}(t)=\sum_{j=1}^Nb_{ij}(\theta_{i}(t)-\theta_{j}(t)),\\
&\quad i=1,...,N,\;t=1,...,D
\end{array}
 \end{displaymath}
where $N$ is the number of nodes, $D$ the number of time periods, $p_{i}(t)$ the real power at node $i$ and time $t$, and $\theta_{i}(t)$ the voltage angle at node $i$ and time $t$. The objective, $F(p)$, is the total cost of generation over a sequence of time periods, which we assume to be convex. The first set of constraints enforces nodal power balances and the linearized power flow, and the second set of constraints limits the power produced or consumed at every node. Examples of the latter are generation limits or load levels. Because the above optimization has linear constraints and a convex objective, it is easy to solve at realistic scales encountered in power systems.

A number of studies have recently developed high fidelity representations of flexible load aggregations in the form of storage with time-varying parameters.  For example,~\cite{nayyar_taylor} identifies effective storage models for deferrable load aggregations. Lossless storage with only energy constraints is represented by
\begin{eqnarray*}
&&e_{i}(t+1)=e_{i}(t)+u_{i}(t)\\
&& 0\leq e_{i}(t)\leq S_{i}(t)
\end{eqnarray*} 
where $e_{i}(t)$ is the state of charge, $u_{i}(t)$ the power injection or extraction, and $S_{i}(t)$ the energy capacity at storage $i$ and time $t$. Observe that this storage models fits seamlessly within multi-period optimally power flow because the constraints are linear.

A number of DR resources could also be represented in optimal power flow this way. For example, any of the polytope models of Section~\ref{sec:load_models} could be straightforwardly inserted into a multiperiod optimal power flow. However, such an approach could introduce millions of new variables and constraints, which would be unwieldy for system operators to manage and difficult for load aggregators to communicate to system operators, e.g., as part of a bidding process. This is the motivation for representing load aggregations as generalized storage in~\cite{mathieu2013energy,mathieu2015arbitraging,nayyar_taylor,alizadeh_loadmodels}. However, this approach is also restrictive because aggregations of some load types may not be well represented as storage.

In this paper, we seek general polytope representations of load aggregations of the form $\mathbf{P}=\{x\;|\;Ax\leq b\}$, i.e., a small number of linear constraints (which we quantify Section~\ref{sec:algorithm}). Here, $x\in \mathbb{R}^D$ is the vector of power injections into the aggregation through time. Since $\mathbf{P}$ is also a (small) polytope, it can be straightforwardly added to the above multi-period optimal power flow without adding a large number of variables and constraints, thus preserving its computational tractability.

\subsection{Modeling Loads as Polytopes}
\label{sec:load_models}

In this section, we survey commonly known H-representations of polytope descriptions for several standard load types. It is natural that we confine our attention to bounded polytopes because loads cannot consume infinite power over a finite number of time periods. For simplicity of exposition, we assume that the duration of each time period is one. The (constant) power use by a load over $D$ time periods is represented as a vector of power injections $x \in \mathbb{R}^D$.

We now define some basic quantities that appear in multiple load types. Denote (time-varying) maximum and minimum power limits as $P_\text{max}(i)$ and $P_\text{min}(i)$. We use $S$ to represent the maximum energy usable by the load (or energy storable by the load), and $S_0$ to represent the initial energy stored by the load. We define a dissipation constant $\alpha$ to model losses of stored energy over one time period. Finally, we make use of input and output efficiencies $\eta_\text{in}$ and $\eta_\text{out}$. These efficiencies may represent losses between a load and the electric grid, e.g., AC to DC conversion losses during electric vehicle charging.

\subsubsection{Storage-like Loads}
\label{sec:storage_loads}
We first consider loads modeled by storage that have energy and power limits, leakage losses, and conversion inefficiencies, for instance, a charging electric vehicle (see, e.g.,~\cite{taylor2015COPS}). We break the power flow $x$ into the components  $x_{\text{in}}$ and $x_{\text{out}}$ which are power flows into and out of the load, respectively. The energy constraint is written as:
\[ 0 \le \alpha^j S_0 + \sum_{i=1}^{j} \alpha^{j-i} \eta_\text{in} x_{\text{in}}(i) + \sum_{i=1}^{j} \alpha^{j-i} \eta_\text{out} x_{\text{out}}(i) \le S, \]
$ 1  \le j \le D $. The power constraints are simply: 
\[ 0 \le x_{\text{in}}(i) \le P_\text{max}(i) \text{ and }  P_\text{min}(i) \le x_{\text{out}}(i) \le 0. \]
Define the matrix
\[ \Gamma = \begin{bmatrix}
1 & 0 & 0 & \dots & 0 \\
\alpha & 1 & 0 & \dots & 0 \\
\alpha^2 & \alpha & 1 & \dots & 0 \\
\vdots & & & & \vdots \\
\alpha^{D-1} & \alpha^{D-2} & \alpha^{D-3} & \dots & 1 \\
\end{bmatrix}. \]
Then, the polytope is defined by the matrices
\[ A_{1} = \begin{bmatrix}
I & 0 \\
-I & 0 \\
0 & I \\
0 & -I \\
\eta_\text{in} \Gamma & \eta_\text{out} \Gamma \\
- \eta_\text{in} \Gamma & - \eta_\text{out} \Gamma \\
\end{bmatrix} \text{ and }
b_{1} = \begin{bmatrix}
P_{max} \\
0 \\
0 \\
-P_{min} \\
S - \alpha S_0  \\
S - \alpha^2 S_0  \\
\vdots \\
S - \alpha^D S_0  \\
\alpha S_0  \\
\alpha^2 S_0  \\
\vdots \\
\alpha^D S_0  \\
\end{bmatrix}; \]
In this case, we explicitly write the polytope as
\[\mathbf{P}_1 = 
\left \{ \begin{bmatrix} x_{\text{in}} \\ x_{\text{out}}  \end{bmatrix} \, \left| \,
A_{1} 
\begin{bmatrix} x_{\text{in}} \\ x_{\text{out}}  \end{bmatrix} \leq b_{1}\right. \right \}. \]


\subsubsection{Thermostatic loads}
\label{sec:thermo_load_model}

Thermostatic loads (TCLs) are modeled in~\cite{hao_thermostatic}, which shows how to map parameters associated with TCLs to those associated with generalized loads. The authors specify TCLs in terms of a set of parameters $\chi^k = (a,b,\theta_a, \theta_r, \Delta, P_m)$, where $a = \frac{1}{R C}$, $b = \frac{\eta}{C}$, $R$ is thermal resistance, $C$ is thermal capacitance, $P_m$ is rated electrical power, $\eta$ is coefficient of performance, $\theta_a$ is ambient temperature, $\theta_r$ is the set-point temperature, and $\Delta$ is the dead-band.

For a TCL, the dynamics are written in terms of the temperature $\theta(t)$ as follows:
\[ \theta(t + 1) = (1-a) \theta(t) + a \theta_a - b x(t). \]
We can expand this equation as:
\[
\theta(j) = (1-a)^j \theta_0 + a  \sum_{i = 1}^{j} \theta_a (i) (1-a)^{j - i} - b \sum_{i = 1}^{j} (1-a)^{j-i} x(j).
\]
The temperature deadband constraint is then given by:
\[ \theta_r - \Delta \le \theta(j) \le \theta_r + \Delta; \, 1  \le j \le D. \]

Let us denote $(1-a)^j \theta_0 + a  \sum_{i=0}^{j-1} \theta_a (i) (1-a)^i$ as $\theta_j$. Then, the deadband constraint is equivalently stated as:
\begin{multline*} \frac{\theta_r - \Delta - \theta_j}{b} \le - \sum_{i = 1}^{j} (1-a)^{j-i} x(j) \le \frac{\theta_r + \Delta -\theta_j}{b}, \\ \, 1  \le j \le D. 
\end{multline*} 
The above inequality is similar to the energy constraint of a generalized storage load, and can be similarly written in H-representation.

\subsubsection{Deferrable loads}
\label{sec:def_loads}

Deferrable loads like electric vehicles are essentially storage-like loads with arrival and departure times. In this example, we present perfectly efficient deferrable loads, which have a power constraint and a single equality energy constraint~\cite{nayyar_taylor}. We denote the total energy requirement of the load by $E$.

The constraints for a deferrable load may be written as:
\[ 0 \le x(i) \le P_\text{max}(i), \, 1  \le i \le D \text{ and}
\sum_{i = 1}^{D}x(i) = E.\]
The associated matrix representation is:
\[ A_{1} = \begin{bmatrix}
& \mathrm{I} \\
& -\mathrm{I} \\
1 & \dots & \ 1 \\
-1 & \dots & -1
\end{bmatrix} \text{ and } 
b_{1} = \begin{bmatrix}
P_\text{max} \\
0 \\
E \\
-E
\end{bmatrix}. \]

Arrival and departure constraints are encoded in the vector $P_{\max}$ by setting
\[ P_{\max}(i) = 0 \, \text{ for } \, i < t_a \text{ or } i \ge t_d, \]
where $t_a$ is the arrival time and $t_d$ the departure time.

\subsubsection{Differential power constraints}
Differential power constraints can be used to prevent large changes in the power consumption or supply of a load, and are commonly encountered when dealing with industrial equipment. They may be added into any of the above load models. We denote the maximum allowed bi-directional difference between the power used in a period and the power used is a subsequent period as $\delta>0$. The differential power constraints may be written as below, and a matrix formulation is easily derived.
\[ -\delta \le x(i+1) - x(i) \le \delta \text{ for } 1 \le i \le D-1.  \]


\subsubsection{Non-polytopic loads}
Finally, it is worth discussing a type of load that does not have a polytope formulation. Consider a load which must use 100 kW of power for a one hour period during a specified three-hour window. We can represent this load as the union of three points in $\mathbb{R}^3$: $ \{(100, 0, 0), \, (0, 100, 0), \, (0, 0, 100)\}$. Obviously, the resultant set is non-convex, and would typically be represented with integer constraints. Such a load cannot be simply aggregated with other polytopes using our subsequent approach. However, polytopic or other convex relaxations of such load models can often be constructed. For instance, the above example can be relaxed to $\sum_{i=1}^3x(i)=100$, $0\leq x(i)\leq 100$ for $i=1,2,3$, which is a polytope.

\subsection {Load Aggregation as Minkowski Sums}
\label{sec: Msum}

Individual loads in DR programs are generally small compared to the size of resources normally dispatched by system operators. As discussed in the Introduction and Section~\ref{sec:role}, adding potentially $10^6$ small loads to the scope of their responsibilities is undesirable. Aggregators act as intermediaries, finding a single compact representation of these loads for the system operator and then controlling the loads in response to the system operator's instructions. For loads specified
as polytopes, their aggregate capability is exactly described by the Minkowski sum, as observed in~\cite{hao_thermostatic}; in~\cite{alizadeh_loadmodels}, this quantity is referred to as the plasticity of the aggregation.

The Minkowski sum of two polytopes in $\mathbb{R}^D$,  $\mathbf{P}_1$ and $\mathbf{P}_2$, is itself a polytope defined by
 \begin{equation}
\mathbf{P}_3 = \{ z \, | \, z = x + y, \, x \in \mathbf{P}_1, \, y \in \mathbf{P}_2 \}
 \end{equation}
In words, if $\mathbf{P}_1$ and $\mathbf{P}_2$ are the sets of feasible power profiles of two loads, $\mathbf{P}_3$ is the set of feasible power profiles of the aggregation of the two loads.

If the polytopes have V-Representations $\bar{X}$ and $\bar{Y}$ respectively, then the V-Representation of the Minkowski Sum can be found by taking the sum of each vertex pair $\{x + y \;|\; x \in \bar{X}, y \in \bar{Y}\}$, and finding the convex hull of the result.

However, if the polytopes are specified in H-representation, the above method is computationally intractable for non-trivial polytopes. This is because it requires performing the vertex enumeration operation for both polytopes. As discussed in Section~\ref{sec:notation}, no known polynomial time algorithm exists for vertex enumeration~\cite{boros_nphard,khachiyan_generating_hard}.

\section {Approximate Load Aggregation}\label{mainresult}
We now develop a generic outer approximation of the Minkowski sum of two polytopes.

\subsection{Polytopes with the same shape}

Consider the following pair of polytopes in H-representation:
\[
\mathbf{P}_1 = \{ x \, | \, A_1 x \leq b_1\} \quad \textrm{and}\quad  \mathbf{P}_2 = \{ y \, | \, A_2 y \leq b_2\} .
\]

We would like to find an approximate representation for the polytope $ \mathbf{P}_3 = \{z \, | \, z = x + y, x \in \mathbf{P}_1, y \in \mathbf{P}_2 \} $, the Minkowski sum of $\mathbf{P}_1 \text{ and } \mathbf{P}_2 $.

\begin{theorem} [Outer Approximation]\label{prop:outer}
Suppose $A_1=A_2=A$. The polytope $ \mathbf{P}_4 = \{ z \, |\, A z \leq b_1 + b_2 \} \subset\mathbb{R}^D$ is an outer approximation to $ \mathbf{P}_3$, the Minkowski sum of $ \mathbf{P}_1 \text{ and } \mathbf{P}_2$.
\end{theorem}

\begin{proof}
Suppose $z$ is in the Minkowski sum of $\mathbf{P}_1$ and $\mathbf{P}_2$. Then there exist $x_1\in \mathbf{P}_1$ and $x_2\in \mathbf{P}_2$ such that $z=x_1+x_2$. Adding the constraints 
\[
A x_1 \leq b_1 \quad\textrm{and}\quad A x_2 \leq b_2,
\]
we obtain $A(x_1+x_2)\leq b_1+b_2$. Therefore, $z\in\mathbf{P}_4$. Since any element of the Minkowski sum of $\mathbf{P}_1$ and $\mathbf{P}_2$ is in $\mathbf{P}_4$, it is an outer approximation.
\end{proof}
We will refer to the polytope $\mathbf{P}_4 =  \{ z \, | \, A z \leq (b_1 + b_2) \} $ as the \textit{outer Minkowski approximation}. We remark that the outer Minkowski approximation could also be referred to as a relaxation of the exact Minkowski sum.

{\bf Example 1:} Consider the polytopes $\mathbf{P}_1$, $\mathbf{P}_2$ and $\mathbf{P}_3$ shown in Figure \ref{Fig:two}, where $\mathbf{P}_3$ is the Minkowski Sum of $\mathbf{P}_1$ and $\mathbf{P}_2$. All are triangles in $\mathbb{R}^2$. In, V-representation, $\bar{X}_{1} = \{(1,1), (2,1), (1,2)\} $, $\bar{X}_{2} = \{(2,1), (4,1), (2,3)\} $ and
$\bar{X}_{3} = \{(3,2), (6,2), (3,5)\} $. The reader can see that the vertices of $\mathbf{P}_3$ are the sum of vertices of $\mathbf{P}_1$ and $\mathbf{P}_2$, and that other points generated by the sum of points inside $\mathbf{P}_1$ and $\mathbf{P}_2$ lie within $\mathbf{P}_3$.
\begin{figure} [h]
\includegraphics[width=9cm]{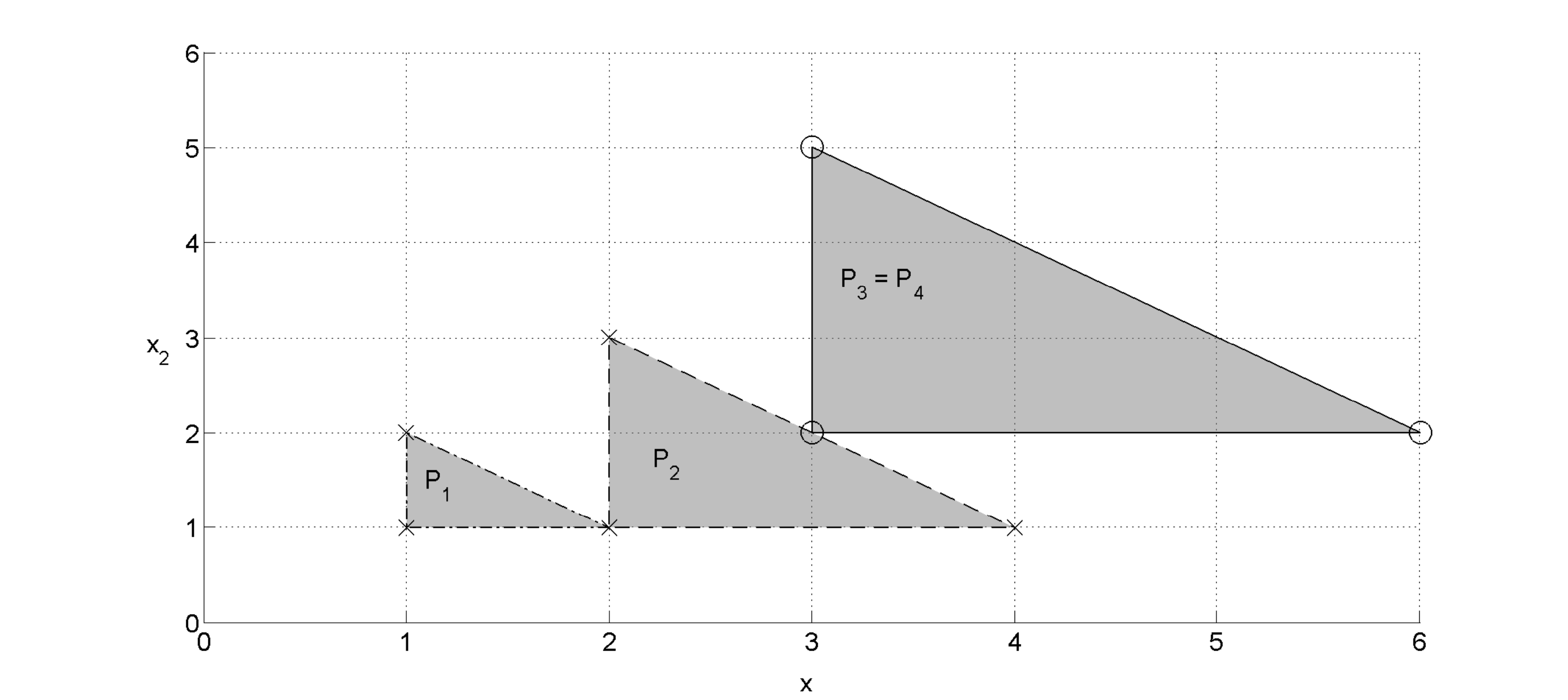}
\caption{The Minkowski sum of two triangles.}
\label{Fig:two}
\end{figure}
The H-representations of $\mathbf{P}_1$ and $\mathbf{P}_2$ are:
\[ A_{1} = \begin{bmatrix}
-1 & 0 \\
0 & -1 \\
1 & 1\\
\end{bmatrix}, 
b_{1} = \begin{bmatrix}
-1 \\
-1 \\
3
\end{bmatrix}\] 
\[ A_{2} = \begin{bmatrix}
-1 & 0 \\
0 & -1 \\
1 & 1\\
\end{bmatrix}, 
b_{2} = \begin{bmatrix}
-2 \\
-1 \\
5
\end{bmatrix}. \]
Since $A_{1} = A_{2} = A$, Proposition~\ref{prop:outer} may be used to find the outer approximation, which we denote $\mathbf{P}_4$ and is given in H-representation by
\[ A = \begin{bmatrix}
-1 & 0 \\
0 & -1 \\
1 & 1\\
\end{bmatrix}, 
b_{4} = \begin{bmatrix}
-3 \\
-2 \\
8
\end{bmatrix}.\] 
It can be verified that this is the exact Minkowski sum.

\subsection {Extension to general polytopes}
\label{sec:gen_polytopes}

The above approximation is limited to Minkowski Sums of polytopes that have the same $A$-matrices, which restricts its applicability to aggregations of loads of the same type. We now extend this formulation to arbitrary polytopes in $\mathbb{R}^D$, which broadens its applicability to aggregations containing many different types of loads.

Consider two polytopes in their minimum H-Representation, $\mathbf{P}_1$ and $\mathbf{P}_2$, described by the matrix-vector pairs $(A_{1}, b_{1})$ and $(A_{2}, b_{2})$. An exact, alternate H-representation for $\mathbf{P}_1$ and $\mathbf{P}_2$ can be constructed in terms of the matrix-vector pairs $(A', b'_{1})$ and $(A', b'_{2})$, where $A'$, $b'_{1}$, and $b'_{2}$ are new matrices which we describe below.

Observe that $\mathbf{P}_1$ can be described in set notation as an intersection of half-spaces, each of which is defined by a linear inequality:
\[
\mathbf{P}_1=\bigcap_{i=1}^N \{x \, | \, a_1(i)^T x \leq \mathrm{b}_1(i) \},
\]
where
\[ A_1 = 
\begin{bmatrix}
  a_1(1)^T \\
  \vdots \\
  a_1(N)^T \\
\end{bmatrix} \text{ and } 
b_1 = 
\begin{bmatrix}
  \mathrm{b}_1(1) \\
  \vdots \\
  \mathrm{b}_1(N)
\end{bmatrix}. \]
 
Here, $  a_1(1)^T \dots a_1(N)^T$ are row-vectors in $\mathbb{R}^D $ and $\mathrm{b}_1(1) \dots \mathrm{b}_1(N)$ are scalars. From this expression, we see that:
\begin{itemize}
\item The rows of the matrix-vector pair $(A_{1}, b_{1})$ can be arbitrarily reordered without changing the polytope.
\item We can add an additional linear constraint to the polytope (i.e., an additional row to the matrix-vector pair), $a(N+1)^T x \leq b(N+1) $,  provided that the following inclusion is satisfied:
\end{itemize}
\begin{equation}
\{x \, | \, A_1 x \leq b_1\} \, 
\subseteq \,
\{x \, | \,  a(N+1)^T x \leq b(N+1)\}\label{eq:inc}
\end{equation}
Equation~(\ref{eq:inc}) states that the polytope $ \mathbf{P}_1 $ lies inside the half-space defined by $ a(N+1)^T x \leq b(N+1) $. We refer to such inequality constraints as \textit{redundant constraints} because they can be added to or eliminated from a polytope without changing it~\cite{Avis}.

Our subsequent approximation attains the highest accuracy when redundant constraints with the smallest possible $b(N+1)$ are used. For an arbitrary row-vector, $ a(\mathrm{N}+1)^T $, we can find the smallest constant $b(\mathrm{N}+1)^*$ that satisfies Equation ~(\ref{eq:inc}) by solving the linear program:    
\begin{equation}
\begin{array}{lllc}

b(N+1)^* & = & \text{maximize} & a(N+1)^T x \\
& & \text{subject to} & A_1 x \leq b_1 

\end{array}\label{eq:red}
\end{equation}
For this choice of $b(N+1)^*$, the equality $ a(N+1)^T x = b(N+1)^* $ describes a hyperplane that is tangent to $\mathbf{P}_1$.

Thus, if a constraint $ a_2(M)^T x \leq  b_2(M) $ is present in the H-representation of polytope $\mathbf{P}_2$ but not $\mathbf{P}_1$, we can add it as the $N+1^\text{th}$ row in $A_1$, and find the associated scalar $ b_1(N+1)^* $ using~(\ref{eq:red}) (or vice-versa). This constraint will then be tangent to polytope $\mathbf{P}_1$.

By adding redundant constraints as described above and reordering, we construct alternate representations for polytopes $\mathbf{P}_1$ and $\mathbf{P}_2$ as the matrix-vector pairs $(A', b_{1}')$ and $(A', b_{2}')$. These representations have the same $A$-matrices, and therefore we can obtain their outer Minkowski approximation via Proposition~\ref{prop:outer}.

It should be noted that if polytopes $\mathbf{P}_1$ and $\mathbf{P}_2$ have $m_1$ and $m_2$ constraints, respectively, with $c$ constraints in common, then the outer Minkowski approximation will have $m_1 + m_2 - c$ constraints; i.e. it's A-matrix will have $m_1 + m_2 - c$ rows.

\subsection{Load aggregation algorithm}
\label{sec:algorithm}
We now present our procedure as an algorithm for approximately representing aggregations of loads described by polytopes.

\begin{enumerate}
\item \textit{Input:} $N$ load polytopes over $D$ time periods in H-representation: $\mathbf{P}_1, \mathbf{P}_2, \dots, \mathbf{P}_N \in \mathbb{R}^D$. Each polytope is described by an arbitrary number of constraints.
\item Search through the $A$-matrices of all $N$ polytopes and make a list of every unique row. This is a polynomial-time sorting operation. The $A'$ matrix consists of all unique rows, and is in $\mathbb{R}^{c\times N}$.
\item For all $N$ polytopes and all $c$ unique constraints, run linear programs to find tangent facets, and construct the vectors $b_1',...,b_N'$. The total number of linear programs run is upper bounded by $cN $, and can be substantially less if the $A$-matrices contain many common rows.
\item \textit{Output:} By Proposition~\ref{prop:outer}, the polytope $\{x\,|\,A'x\leq \sum_{i=1}^N b_i'\}$ is an outer approximation of the Minkowski sum of the $N$ polytopic loads.
\end{enumerate}

Linear programs have polynomial time complexity~\cite{Bazaraa_LP_and_NF}. As our algorithm invokes a polynomial number of LPs, its complexity also grows polynomially with the number of loads and dimensions. We illustrate the application of the algorithm in the below example.

{\bf Example 3:} Suppose we have two loads with $A$-matrices
\[ A_1 = 
\begin{bmatrix}
  A \\
  a_1^T \\
\end{bmatrix} \quad\textrm{and}\quad A_2 = 
\begin{bmatrix}
  A \\
  a_2^T \\
\end{bmatrix},
\]
and $\mathbf{P}_1=\{x \, | \,  A_1 x \leq b_1\}$ and $\mathbf{P}_2=\{x \, | \,  A_2 x \leq b_2\}$. Suppose further that $b_1\in\mathbb{R}^N$ and $b_2\in\mathbb{R}^N$, and define
\[
\begin{array}{lllc}

b_1(N+1)^* & = & \text{maximize} & a_2^T x \\
& & \text{subject to} & A_1 x \leq b_1 

\end{array}
\]
$b_2(N+1)^*$ is defined analogously. Let
\[
A'=\begin{bmatrix}
  A \\
  a_1^T \\
  a_2^T \\
\end{bmatrix},\, 
b_1'=\begin{bmatrix}
  b_1 \\
  b_1(N+1)^* \\
\end{bmatrix},\,
b_2'=\begin{bmatrix}
  b_2(1) \\
  \vdots\\
  b_2(N-1)\\
  b_2(N+1)^* \\
  b_2(N)
\end{bmatrix}.
\]
Then $\mathbf{P}_1=\{x \, | \,  A' x \leq b_1'\}$ and $\mathbf{P}_2=\{x \, | \,  A' x \leq b_2'\}$. Using Proposition~\ref{prop:outer}, we obtain $\mathbf{P}_4=\{x \, | \,  A' x \leq b_1'+b_2'\}$ as an outer approximation of the Minkowski sum of $\mathbf{P}_1$ and $\mathbf{P}_2$.



\section{Examples}\label{sec:examples}

\subsection{Numerical Examples}

In this section, we numerically evaluate the accuracy of the outer Minkowski approximation for two general classes of loads, thermostatic loads and generalized energy storage. As our load aggregations are closed polytopes, they can be characterized by volumes. The outer Minkowski approximation contains the exact Minkowski sum and therefore always has larger volume; when their volumes are identical, the approximation is exact. The ratio of volumes of two polytopes hence measures absolute accuracy when one polytope is the exact Minkowski sum, and relative accuracy when both polytopes are approximations. We thus use such volume ratios to describe the error in the outer Minkowski approximation. 

However, the exact computation of volume of a high-dimensional polytope is an NP-hard problem\cite{dyer1998complexity}. We thus make use of a Monte-Carlo method for volume estimation. For the polytopes in question, we define a bounding box in $\mathbb{R}^D$ and uniformly sample this box. The fraction of points inside the polytope yields an estimate of the volume.

\subsubsection{Thermostatic Loads}

Models for thermostatic loads were described in Section~\ref{sec:thermo_load_model}. We generate sets of randomized parameters to describe 1000 distinct loads; the mean values ($\mu$) of the parameters varied are: the thermal capacitance ($2 \text{ kWh/\degree C}$), the thermal resistance ($2 \text{\degree C/kW }$), the rated electrical power ($ 5.6 \text{ kW}$), the coefficient of performance ($2.5$), the temperature setpoint ($22.5 \degree \text{C}$) and the temperature deadband ($0.3 \degree \text{C}$), which are taken from~\cite{hao_thermostatic}. 
Each of the load parameters are drawn from a uniform distribution from between $0.9 \mu - 1.1 \mu$ for a low heterogeneity scenario, and from between $0.8 \mu - 1.2 \mu$ for a high heterogeneity scenario. Additionally the starting temperature of each load is drawn from a uniform distribution over the deadband.

We consider a 1-hour time period and look at the performance of the approximation as the interval of discretization is varied, e.g. two, 30-minute slots, four, 15-minute slots, etc. When computing the outer Minkowski approximation (denoted as OM), each load is first approximated by an equivalent load whose dissipation constant is the mean of the set; this approximation is computed as an outer (necessary) approximation.

We also compute necessary (denoted as GB-N) and sufficient (denoted as GB-S) generalized battery approximations for the aggregation of these loads, as in~\cite{hao_thermostatic}. Note that~\cite{hao_thermostatic} also addresses the control of a collection of thermostatic loads for regulation on very fast timescales, which is beyond our scope; the aggregate models developed in it are useful for comparison given that exact results cannot be obtained, but are intended for a different purpose than our approach. These battery approximations are modeled as polytopes, as explained in Section~\ref{sec:storage_loads}, after which their volumes are found. One billion points are generated for each Monte-Carlo volume estimation case.

In Figure~\ref{OMvsGBFig}, we plot the volume ratios: OM / GB-N and GB-S / GB-N as a function of the number of slots used for discretization of the 1-hour period, for both low (Low-h) and high (High-h) heterogeneity scenarios.  

\begin{figure} [h]
\includegraphics[width=9cm]{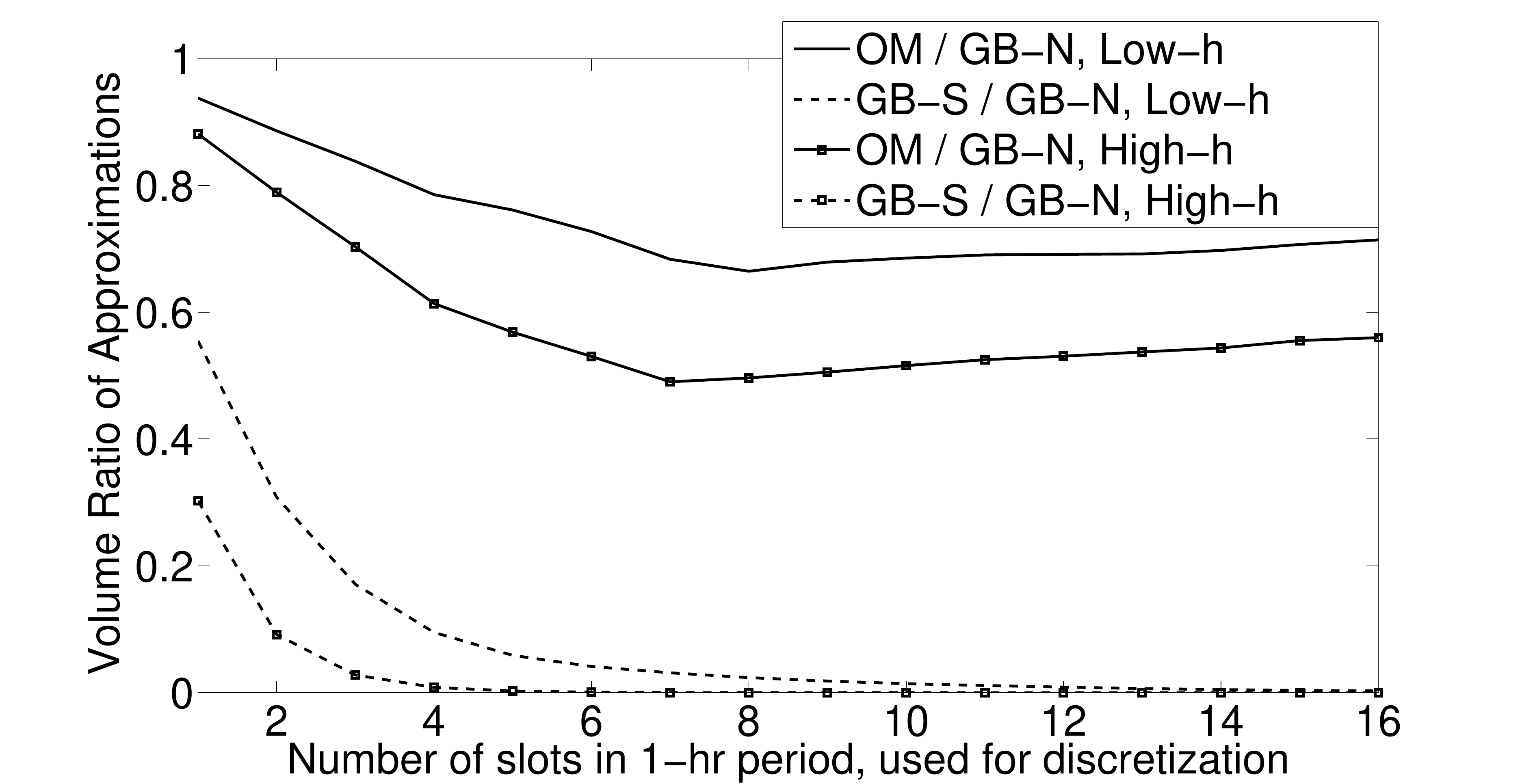}
\caption{Volume comparison of thermostatic load aggregations.}
\label{OMvsGBFig}
\end{figure}

We find that for both scenarios, the size of the OM approximation is smaller than the GB-N from~\cite{hao_thermostatic} (and, by construction, larger than the GB-S). Hence the OM approximation is more accurate than the GB-N approximation, by approximately a factor of $1.5 - 2$ depending on the amount of load heterogeneity. Additionally, we see that the performance of the OM approximation improves vis-a-vis the  GB-N approximation in the higher heterogeneity scenario.

\subsubsection{Storage Loads}

Models for storage loads were described in Section~\ref{sec:storage_loads}. Here, we focus on non-dissipative storage loads that are fully present for the period of aggregation and have input/output efficiencies of unity.

We take randomized parameters for 2000 loads, and use them to compute 1000 pairwise sums (and/or approximations). We carry out this process for dimensions from $\mathbb{R}^2$ to $\mathbb{R}^{20}$ by instantiating loads for time intervals of $D=\{2,...,20\}$ hours, with hourly slots. The loads have power limits uniformly distributed between $30$ and $70$, and energy capacities that are uniformly distributed with between $120$ and $280$; finally, the initial states of charge are uniformly distributed from $0$ to the energy capacity. 

We use MPT~\cite{mpt} to compute the volumes of the approximate and exact pairwise sums up to $\mathbb{R}^6$, beyond which the computations become intractable. We then compute the average over the 1000 cases of the ratio of the exact volume and that obtained by the OM approximation. We also use a Monte-Carlo method to estimate the volume of the approximation up to $\mathbb{R}^{20}$, which we use to validate the results from MPT and to examine the behavior of the approximation with increase in dimension. We comment that approximately $2.74 \%$ of the computed data had to be thrown out because of numerical errors in computations by the MPT toolbox. We plot the results in Figures \ref{MVFig} and \ref{MEFig}.

\begin{figure} [h]
\includegraphics[width=9cm]{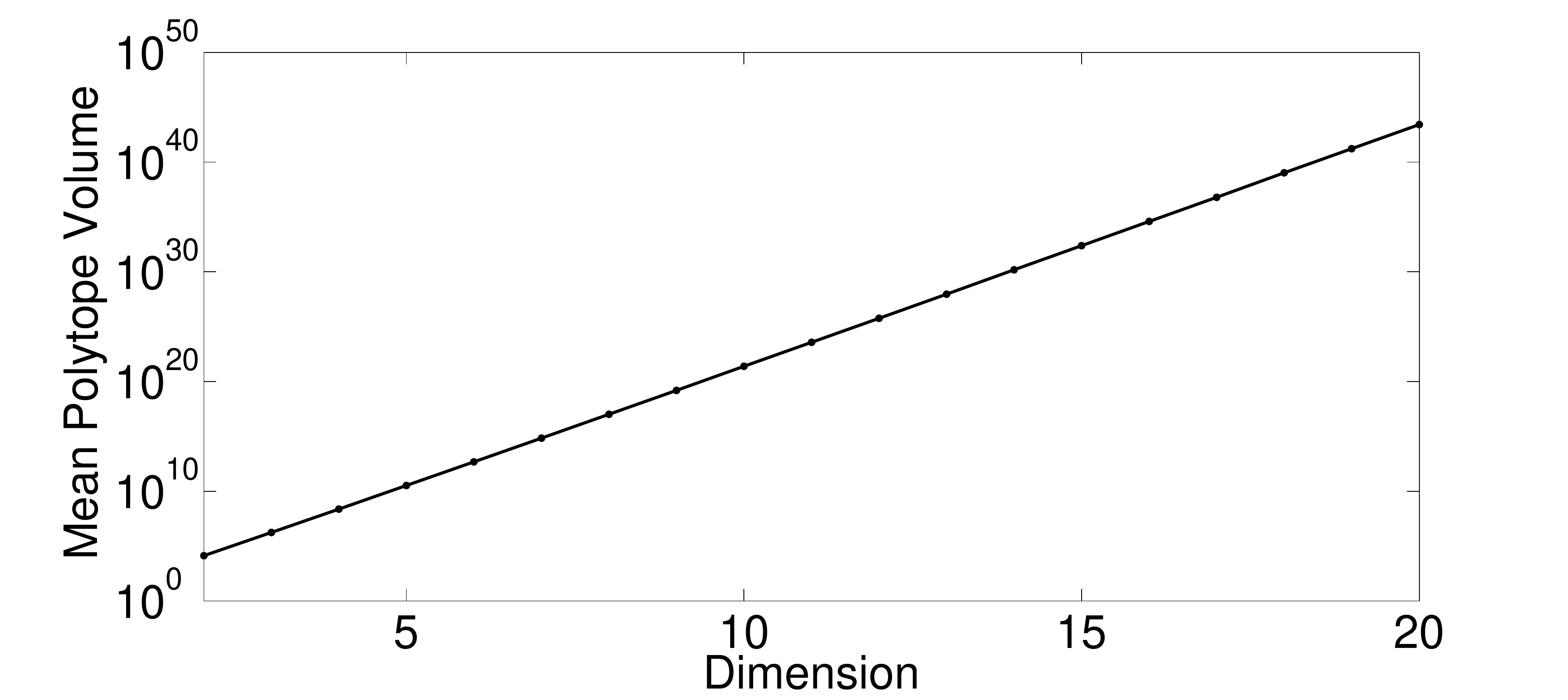}
\caption{Approximation volume for aggregations of storage loads, up to $\mathbb{R}^{20}$.}
\label{MVFig}
\end{figure}

As observed in Figure \ref{MVFig}, the mean volume of the approximate aggregation scales exponentially with dimension as expected; this appears as linear on a semilogarithmic plot.

\begin{figure} [h]
\includegraphics[width=9cm]{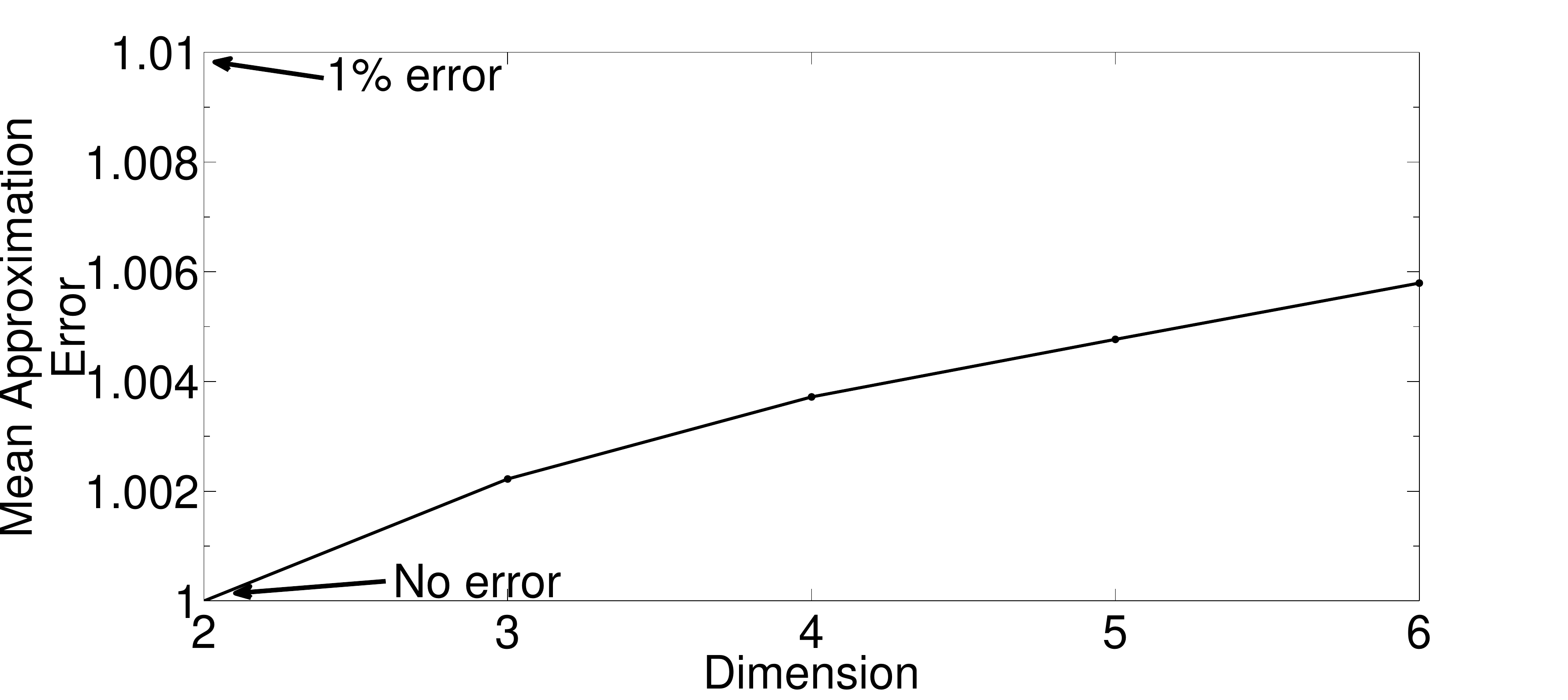}
\caption{Approximation error for aggregations of storage loads, up to $\mathbb{R}^6$.}
\label{MEFig}
\end{figure}

The error (defined as the ratio of the volume of the approximation to the volume of the exact result)is computable only up to $\mathbb{R}^6$. We see, in Figure \ref{MEFig}, that it remains below $0.7 \%$ for those dimensions, and grows sub-linearly, indicating that the OM approximation continues to achieve low errors in higher dimensions. 

\subsection{Analytical Results}

In this section we present two useful analytical results regarding the exactness of the outer Minkowski approximation when applied to specific load classes.

\subsubsection{Loads with only power constraints}

Let us consider loads with power limit vectors in $\mathbb{R}^D$, $P_h$ and $P_l$ for all $D$ time periods, such that $ P_l \le x \le P_h $ (different power limits for each time period). They may be represented by the following simple $D$-dimensional hypercube:
\[  \begin{bmatrix}
\mathrm{I} \\
-\mathrm{I} \\
\end{bmatrix} x \leq \begin{bmatrix}
P_h \\
-P_l \\
\end{bmatrix}. \]

\begin{theorem} [Exactness of outer approximation for hypercubes] Consider two hypercube loads defined by power limit vectors, $P_{h1}$ and $P_{l1}$ for the first and $P_{h2}$ and $P_{l 2}$ for the second. The outer Minkowski approximation to the Minkowski sum of these loads is exact, and is given by:
\[ \begin{bmatrix}
\mathrm{I} \\
-\mathrm{I} \\
\end{bmatrix} x \leq \begin{bmatrix}
P_{h 1} + P_{h 2} \\
P_{l 1} + P_{l 2} \\
\end{bmatrix}. \]

\end{theorem}

\begin{proof}
The exact Minkowski sum of two hypercubes can be computed by taking the convex hull of the sums of all vertex pairs. Straightforward calculation gives the vertex set $\bigcup_{i=1}^D(P_l(i)+P_h(i))$. The outer Minkowski approximation is the same hypercube.
\end{proof}

\subsubsection{Deferrable loads}

Let us consider deferrable loads whose total energy consumption is denoted as $E$, as in Section~\ref{sec:def_loads}. Such a load requires nonnegative power over all $D$ time periods, and its total energy consumption must be $E$ by the last time period. The $k^{\text{th}}$ such load may be represented by the following matrices:
\[ A = \begin{bmatrix}
& -\mathrm{I} &  \\
1 &  \dots & 1 \\
-1 &  \dots & -1 \\
\end{bmatrix}, 
b_k= \begin{bmatrix}
0 \\
E_k \\
-E_k \\
\end{bmatrix}. \]


\begin{theorem} [Exactness of outer Minkowski approximation for deferrable load polytopes] Consider two deferrable loads defined as above with energy requirements  $E_1$ and $E_2$ and which are present over the same time periods. Then, their outer Minkowski approximation is exact and is given by $\mathbf{P} = \{ x \, | \, A x \leq b' \}$, where
\[ A= \begin{bmatrix}
& -\mathrm{I} & \\
1 &  \dots & 1 \\
-1 & \dots & -1 \\
\end{bmatrix}, \,
b' = \begin{bmatrix}
0 \\
E_1 + E_2\\
-(E_1 + E_2) \\
\end{bmatrix}.\]

\end{theorem}

\begin{proof}
Load $i$ is a standard simplex defined by the hyperplane $x_1 + \dots + x_n = E_i$. It has $D$ vertices which are simply $\{(E_i, 0, \dots, 0), (0, E_i, \dots, 0), \dots, (0, 0, \dots, E_i)\}$.

The outer Minkowski approximation of the two loads is given by the matrices:
\[ A_{3} = \begin{bmatrix}
& & -\mathrm{I} \\
1 & 1 & \dots & 1 \\
-1 & -1 & \dots & -1 \\
\end{bmatrix}, 
b_{3} = \begin{bmatrix}
0 \\
E_1 + E_2\\
-(E_1 + E_2) \\
\end{bmatrix}. \]

The Minkowski sum retains the same structure as the base polytopes, and hence it's V-representation is simply $\{(E_1 + E_2, 0, \dots, 0),  (0, E_1 + E_2, \dots, 0), \dots, (0, 0, \dots, E_1 + E_2)\}$. The Minkowski sum is exact in this case, as can be verified by taking the convex hull of the pairwise sum of vertices from the two polytopes.
\end{proof}

This result is similar to that in\cite{nayyar_taylor}, which develops an exact storage representation for deferrable loads with only energy constraints and arbitrary arrival and departure times.

\section{Conclusions}
We have developed a technique for aggregating populations of heterogeneous loads described by polytopes. The approach is powerful because it captures a wide range of load types, is computationally tractable, and theoretically and empirically accurate in scenarios of practical interest.


We are currently developing several extensions that incorporate uncertainty, thus enabling the aggregate representation of probabilistically defined loads via similar techniques. This is important as many types of resources may be unable to exactly specify their constraints, e.g., the arrival and departure times of electric vehicles.

Finally, another important area of work concerns the problem of resource selection. While this work describes how to aggregate the available flexibility from a collection of DR resources, it does not offer specify how to allocate the bulk power into and out of an aggregation amongst individual loads. Doing so necessitates balancing considerations of equity, cost, and maintaining maximal flexibility for future time periods.

\bibliographystyle{IEEEtran}
\bibliography{SuhailBib}

\end{document}